\documentclass[a4paper,12pt]{amsart}
\usepackage[all]{xy}
\usepackage{amssymb,upref}

\newtheorem{thm}{Theorem}[section]
\newtheorem{cor}[thm]{Corollary}
\newtheorem{lem}[thm]{Lemma}
\newtheorem{prop}[thm]{Proposition}

\theoremstyle{definition}

\newtheorem{ntn}[thm]{Notation}

\theoremstyle{remark}
\newtheorem{rmk}[thm]{Remark}

\newtheorem{example}[thm]{Example}

\newcommand{\lsp}{\operatorname{span}}
\newcommand{\clsp}{\operatorname{\overline{\lsp}}}

\newcommand{\midtext}[1]{\quad\text{ #1 }\quad}
\renewcommand{\and}{\midtext{and}}

\newcommand{\N}{\mathbb N}
\newcommand{\Z}{\mathbb Z}

\newcommand{\T}{\mathbb T}

\newcommand{\CC}{\mathcal C}

\newcommand{\GG}{\mathcal G}

\newcommand{\Uu}{\mathcal U}

\renewcommand{\a}{\alpha}

\renewcommand{\d}{\delta}

\newcommand{\m}{\mu}

\renewcommand{\t}{\theta}

\renewcommand{\L}{\Lambda}

\DeclareMathOperator{\aut}{Aut} \DeclareMathOperator{\ad}{Ad}

\newcommand{\id}{\text{\textup{id}}}

\newcommand{\iso}{\overset{\cong}{\longrightarrow}}

\newcommand{\what}{\widehat}

\newcommand{\rt}{\textup{rt}}
\newcommand{\invlim}{\varprojlim}
\newcommand{\dirlim}{\varinjlim}
\newcommand{\tgrphlim}{{
    \renewcommand{\leftarrow}{\leftharpoondown}
    \invlim
}}
\def\cs#1{{
\ensuremath{\mathbin{\stackrel{#1}{\leftharpoondown}}}}}

\title
[Coverings, skew-products and coactions]
{Coverings of skew-products and crossed products by
coactions}
\author{David Pask}
\address{David Pask \\ School of Mathematics and Applied Statistics \\ University of Wollongong \\ NSW, 2522 \\ AUSTRALIA}
\email{dpask@uow.edu.au}
\author{John Quigg}
\address{John Quigg \\ Department of Mathematics and Statistics \\ Arizona State University \\ Tempe, Arizona, 85287 \\ USA}
\email{quigg@asu.edu}
\author{Aidan Sims}
\address{Aidan Sims \\ School of Mathematics and Applied Statistics \\ University of Wollongong \\ NSW, 2522 \\ AUSTRALIA}
\email{asims@uow.edu.au}
\date{June 3, 2007; minor corrections October 2007, March 2008.}
\thanks{This research was supported by the ARC}
\subjclass[2000]{Primary: 46L05; Secondary: 46L55}
\keywords{$C^*$-algebra, coaction, covering, crossed-product, graph algebra, $k$-graph}

\numberwithin{equation}{section}

\begin{document}

\maketitle

\begin{abstract}
Consider a projective limit $G$ of finite groups $G_n$. Fix a
compatible family $\delta^n$ of coactions of the $G_n$ on a
$C^*$-algebra $A$. From this data we obtain a coaction $\delta$
of $G$ on $A$. We show that the coaction crossed product of $A$
by $\delta$ is isomorphic to a direct limit of the coaction
crossed products of $A$ by the $\delta^n$.

If $A = C^*(\Lambda)$ for some $k$-graph $\Lambda$, and if the
coactions $\delta^n$ correspond to skew-products of $\Lambda$,
then we can say more. We prove that the coaction
crossed-product of $C^*(\Lambda)$ by $\delta$ may be realised
as a full corner of the $C^*$-algebra of a $(k+1)$-graph. We
then explore connections with Yeend's topological higher-rank
graphs and their $C^*$-algebras.
\end{abstract}

\section{Introduction}

In this article we investigate how certain coactions of
discrete groups on $k$-graph $C^*$-algebras behave under
inductive limits. This leads to interesting new connections
between $k$-graph $C^*$-algebras, nonabelian duality, and
Yeend's topological higher-rank graph $C^*$-algebras.

We consider a particularly tractable class of coactions of
finite groups on $k$-graph $C^*$-algebras. A functor $c$ from a
$k$-graph $\Lambda$ to a discrete group $G$ gives rise to two
natural constructions. At the level of $k$-graphs, one may
construct the skew-product $k$-graph $\Lambda \times_c G$; and
at the level of $C^*$-algebras, $c$ induces a coaction $\delta$
of $G$ on $C^*(\Lambda)$. It is a theorem of \cite{PQR2} that
these two constructions are compatible in the sense that the
$k$-graph algebra $C^*(\Lambda \times_c G)$ is canonically
isomorphic to the coaction crossed-product $C^*$-algebra
$C^*(\Lambda) \times_\delta G$.

The skew-product construction is also related to discrete
topology: given a regular covering map from a $k$-graph
$\Gamma$ to a connected $k$-graph $\Lambda$, one obtains an
isomorphism of $\Gamma$ with a skew-product of $\Lambda$ by a
discrete group $G$ \cite[Theorem~6.11]{PQR2}. Further results
of \cite{PQR2} then show how to realise the $C^*$-algebra of
$\Gamma$ as a coaction crossed product of the $C^*$-algebra of
$\Lambda$.

The results of \cite{KPS} investigate the relationship between
$C^*(\Lambda)$ and $C^*(\Gamma)$ from a different point of
view. Specifically, they show how a covering $p$ of a $k$-graph
$\Lambda$ by a $k$-graph $\Gamma$ induces an inclusion of
$C^*(\Lambda)$ into $C^*(\Gamma)$. A sequence of compatible
coverings therefore gives rise to an inductive limit of
$C^*$-algebras. The main results of \cite{KPS} show how to
realise this inductive limit as a full corner in the
$C^*$-algebra of a $(k+1)$-graph.

We can combine the ideas discussed in the preceding three
paragraphs as follows. Fix a $k$-graph $\Lambda$, a projective
sequence of finite groups $G_n$, and a sequence of functors
$c_n : \Lambda \to G_n$ which are compatible with the
projective structure. We obtain from this data a sequence of
skew-products $\Lambda \times_{c_n} G_n$ which form a sequence
of compatible coverings of $\Lambda$. By results of \cite{KPS},
we therefore obtain an inductive system of $k$-graph
$C^*$-algebras $C^*(\Lambda \times_{c_n} G_n)$. The results of
\cite{PQR2} show that each $C^*(\Lambda \times_{c_n} G_n)$ is
isomorphic to a coaction crossed product $C^*(\Lambda)
\times_{\delta^n} G_n$. It is therefore natural to ask whether
the direct limit $C^*$-algebra $\varinjlim(C^*(\Lambda
\times_{c_n} G_n))$ is isomorphic to a coaction crossed product
of $C^*(\Lambda)$ by the projective limit group $\varprojlim
G_n$.

After summarising in Section~\ref{sec:prelims} the background
needed for our results, we answer this question in the
affirmative and in greater generality in Theorem~\ref{thm:ccp
continuous}. Given a $C^*$-algebra $A$, a projective limit of
finite groups $G_n$ and a compatible system of coactions of the
$G_n$ on $A$, we show that there is an associated coaction
$\delta$ of $\varprojlim G_n$ on $A$, such that
$
A \times_\delta (\varprojlim G_n)
 \cong \
\varinjlim (A \times_{\delta^n} G_n).
$

In Section~\ref{sec:k-graphs}, we consider the consequences of
Theorem~\ref{thm:ccp continuous} in the original motivating
context of $k$-graph $C^*$-algebras. We consider a $k$-graph
$\Lambda$ together with functors $c_n : \Lambda \to G_n$ which
are consistent with the projective limit structure on the
$G_n$. In Theorem~\ref{thm:tower cong ccp}, we use
Theorem~\ref{thm:ccp continuous} to deduce that $C^*(\Lambda)
\times_\delta G$ is isomorphic to $\varinjlim (C^*(\Lambda)
\times_{\delta^n} G_n)$. Using results of~\cite{KPS}, we
realise $C^*(\Lambda) \times_\delta G$ as a full corner in a
$(k+1)$-graph algebra (Corollary~\ref{cor:tower cong ccp dir
lim}). We digress in Section~\ref{sec:simplicity} to
investigate simplicity of $C^*(\Lambda) \times_\delta G$ via
the results of~\cite{RobSi}.

We conclude in Section~\ref{sec:topological k-graphs} with an
investigation of the connection between our results and Yeend's
notion of a topological $k$-graph \cite{Y1, Y2}. We construct from an
infinite sequence of coverings $p_n : \Lambda_{n+1} \to \Lambda_n$ of
$k$-graphs a projective limit $\Lambda$ which is a topological
$k$-graph. We show that the $C^*$-algebra $C^*(\Lambda)$ of this
topological $k$-graph coincides with the direct limit of the
$C^*(\Lambda_n)$ under the inclusions induced by the $p_n$. In
particular, the system of cocycles $c_n : \Lambda \to G_n$ discussed
in the preceding paragraph yields a cocycle $c : \Lambda \to G :=
\varprojlim(G_n, q_n)$, the skew-product $\Lambda \times_c G$ is a
topological $k$-graph, and the $C^*$-algebras $C^*(\Lambda \times_c
G)$ and $C^*(\Lambda) \times_\delta G$ are isomorphic, generalising
the corresponding result \cite[Theorem~7.1(ii)]{PQR2} for discrete
groups.

\medskip

\section{Preliminaries}\label{sec:prelims}
Throughout this paper, we regard $\N^k$ as a semigroup under
addition with identity element $0$. We denote the canonical
generators of $\N^k$ by $e_1, \dots, e_k$. For $n \in \N^k$, we
denote its coordinates by $n_1, \dots, n_k \in \N$ so that $n =
\sum^k_{i=1} n_i e_i$. For $m,n \in \N^k$, we write $m \le n$
if $m_i \le n_i$ for all $i \in \{1, \dots, k\}$.

We will at times need to identify $\N^k$ with the subsemigroup
of $\N^{k+1}$ consisting of elements $n$ whose last coordinate
is equal to zero. For $n \in \N^k$, we write $(n,0)$ for the
corresponding element of $\N^{k+1}$. When convenient, we regard
$\N^k$ as (the morphisms of) a category with a single object in
which the composition map is the usual addition operation in
$\N^k$.

\subsection{$k$-graphs}
Higher-rank graphs are defined in terms of categories. In this
paper, given a category $\CC$, we will identify the objects
with the identity morphisms, and think of $\CC$ as the
collection of morphisms only. We will write composition in our
categories by juxtaposition.

Fix an integer $k \ge 1$. A \emph{$k$-graph} is a pair
$(\Lambda,d)$ where $\Lambda$ is a countable category and $d :
\Lambda \to \N^k$ is a functor satisfying the factorisation
property: whenever $\lambda \in \Lambda$ and $m,n \in \N^k$
satisfy $d(\lambda) = m + n$, there are unique $\mu,\nu \in
\Lambda$ with $d(\mu) = m$, $d(\nu) = n$, and $\lambda =
\mu\nu$. For $n \in \N^k$, we write $\Lambda^n$ for
$d^{-1}(n)$. If $p \le q \le d(\lambda)$, we denote by
$\lambda(p,q)$ the unique path in $\Lambda^{q-p}$ such that
$\lambda = \lambda' \lambda(p,q) \lambda''$ for some $\lambda'
\in \Lambda^p$ and $\lambda'' \in \Lambda^{d(\lambda) - q}$

Applying the factorisation property with $m = 0$, $n =
d(\lambda)$ and with $m = d(\lambda)$, $n = 0$, one shows that
$\Lambda^0$ is precisely the set of identity morphisms in
$\Lambda$. The codomain and domain maps in $\Lambda$ therefore
determine maps $r,s : \Lambda \to \Lambda^0$. We think of
$\Lambda^0$ as the vertices --- and $\Lambda$ as the paths
--- in a ``$k$-dimensional directed graph.''

Given $F \subset \Lambda$ and $v \in \Lambda^0$ we write $v F$ for $F
\cap r^{-1}(v)$ and $F v$ for $F \cap s^{-1}(v)$. We say that
$\Lambda$ is \emph{row-finite} if $v\Lambda^n$ is a finite set for
all $v \in \Lambda^0$ and $n \in \N^k$, and we say that $\Lambda$ has
\emph{no sources} if $v\Lambda^n$ is always nonempty.

We denote by $\Omega_k$ the $k$-graph $\Omega_k := \{(p,q) \in
\N^k \times \N^k : p \le q\}$ with $r(p,q) := (p,p)$, $s(p,q)
:= (q,q)$ and $d(p,q) := q-p$. As a notational convenience, we
will henceforth denote $(p,p) \in \Omega_k^0$ by $p$. An
\emph{infinite path} in a $k$-graph $\Lambda$ is a
degree-preserving functor (otherwise known as a \emph{$k$-graph
morphism}) $x : \Omega_k \to \Lambda$. The collection of all
infinite paths is denoted $\Lambda^\infty$. We write $r(x)$ for
$x(0)$, and think of this as the range of $x$.

For $\lambda \in \Lambda$ and $x \in s(\lambda)\Lambda^\infty$,
there is a unique infinite path $\lambda x \in
r(\lambda)\Lambda^\infty$ satisfying $(\lambda x)(0,p) :=
\lambda x(0, p-d(\lambda))$ for all $p \ge d(\lambda)$. In
particular, $r(x) x = x$ for all $x \in \Lambda^\infty$, so we
denote $\{x \in \Lambda^\infty : r(x) = v\}$ by
$v\Lambda^\infty$. If $\Lambda$ has no sources, then
$v\Lambda^\infty$ is nonempty for all $v \in \Lambda^0$.

The factorisation property also guarantees that for $x \in
\Lambda^\infty$ and $n \in \N^k$ there is a unique infinite
path $\sigma^n(x) \in x(n)\Lambda^\infty$ such that
$\sigma^n(x)(p,q) = x(p+n, q+n)$. We somewhat imprecisely refer
to $\sigma$ as the \emph{shift map}. Note that
$\sigma^{d(\lambda)}(\lambda x) = x$ for all $\lambda \in
\Lambda$, $x \in s(\lambda)\Lambda^\infty$, and
$x=x(0,n)\sigma^n(x)$ for all $x \in \Lambda^\infty$ and $n \in
\N^k$.

We say a row-finite $k$-graph $\Lambda$ with no sources is
\emph{cofinal} if, for every $v \in \Lambda^0$ and every $x \in
\Lambda^\infty$ there exists $n \in \N^k$ such that $v \Lambda x(n)
\not= \emptyset$. Given $m \not= n \in \N^k$ and $v \in \Lambda^0$,
we say that $\Lambda$ has local periodicity $m,n$ at $v$ if
$\sigma^m(x) = \sigma^n(x)$ for all $x \in v\Lambda^\infty$. We say
that $\Lambda$ has \emph{no local periodicity} if, for every $m,n \in
\N^k$ and every $v \in \Lambda^0$, we have $\sigma^m(x) \not=
\sigma^n(x)$ for some $x \in v\Lambda^\infty$.

\subsection{Skew-products}
Let $\Lambda$ be a $k$-graph, and let $G$ be a group. A
\emph{cocycle} $c : \Lambda \to G$ is a functor from $\Lambda$ to $G$
where the latter is regarded as a category with one object. That is,
$c : \Lambda \to G$ satisfies $c(\mu\nu) = c(\mu)c(\nu)$ whenever
$\mu,\nu$ can be composed in $\Lambda$. It follows that $c(v) = e$
for all $v \in \Lambda^0$, where $e \in G$ is the identity element.

Given a cocycle $c : \Lambda \to G$, we can form the
\emph{skew-product $k$-graph} $\Lambda \times_c G$. We follow
the conventions of \cite[Section~6]{PQR2}. Note that these are
different to those of \cite[Section~5]{KP}. The paths in
$\Lambda \times_c G$ are
\[
(\Lambda \times_c G)^n := \Lambda^n \times G
\]
for each $n \in \N^k$. The range and source maps $r,s : \Lambda
\times_c G \to (\Lambda \times_c G)^0$ are given by
$r(\lambda,g) := (r(\lambda), c(\lambda)g)$ and $s(\lambda, g)
:= (s(\lambda), g)$. Composition is determined by $(\mu,
c(\nu)g)(\nu, g) = (\mu\nu, g)$. It is shown in
\cite[Section~6]{PQR2} that $\Lambda \times_c G$ is a
$k$-graph.

\subsection{Coverings and $(k+1)$-graphs}\label{sec:covering
systems}

We recall here some definitions and results from \cite{KPS} regarding
coverings of $k$-graphs. Given $k$-graphs $\Lambda$ and $\Gamma$, a
$k$-graph morphism $\phi : \Lambda \to \Gamma$ is a functor which
respects the degree maps. A \emph{covering of $k$-graphs} is a triple
$(\Lambda, \Gamma, p)$ where $\Lambda$ and $\Gamma$ are $k$-graphs,
and $p : \Gamma \to \Lambda$ is a $k$-graph morphism which is
surjective and is locally bijective in the sense that for each $v \in
\Gamma^0$, the restrictions $p|_{v\Gamma} : v\Gamma \to p(v)\Lambda$
and $p|_{\Gamma v} : \Gamma v \to \Lambda p(v)$ are bijective.

\begin{rmk}
What we have called a covering of $k$-graphs is a special case of
what was called a ``covering system of $k$-graphs'' in \cite{KPS}. In
general, a covering system consists of a covering of $k$-graphs
together with some extra combinatorial data. We do not need the extra
generality, so we have dropped the word ``system.''
\end{rmk}

A covering $(\Lambda,\Gamma, p)$ is \emph{row-finite} if
$\Lambda$ (equivalently $\Gamma$) is row-finite, and
$|p^{-1}(v)| < \infty$ for all $v \in \Lambda^0$.
Proposition~2.6 of \cite{KPS} shows that we can associate to a
row-finite covering $p : \Gamma \to \Lambda$ of $k$-graphs a
row-finite $(k+1)$-graph $\Lambda \cs{p} \Gamma$ containing
disjoint copies $\imath(\Lambda)$ and $\jmath(\Gamma)$ of
$\Lambda$ and $\Gamma$ with an edge of degree $e_{k+1}$
connecting each vertex $\jmath(v) \in \jmath(\Gamma^0)$ to its
image $\imath(p(v)) \in \imath(\Lambda^0)$.

More generally, given a sequence $(\Lambda_n, \Lambda_{n+1},
p_n)$ of row-finite coverings of $k$-graphs, Corollary~2.10 of
\cite{KPS} shows how to build a $(k+1)$-graph
$\tgrphlim(\Lambda_n; p_n)$, which we sometimes refer to as a
\emph{tower graph}, containing a copy $\imath_n(\Lambda_n)$ of
each individual $k$-graph in the sequence, and an edge of
degree $e_{k+1}$ connecting each $\imath_{n+1}(v) \in
\imath_{n+1}(\Lambda_{n+1}^0)$ to its image $\imath_n(p_n(v))
\in \imath_n(\Lambda_n^0)$. The $(k+1)$-graph
$\tgrphlim(\Lambda_n; p_n)$ has no sources if the $\Lambda_n$
all have no sources.

Given a covering $(\Lambda,\Gamma,p)$, \cite[Proposition~3.2 and
Theorem~3.8]{KPS} show that the covering map $p : \Gamma \to \Lambda$
induces an inclusion $\iota_p : C^*(\Lambda) \to C^*(\Gamma)$. If
$(\L_n,\L_{n+1},p_n)_{n=1}^\infty$ is a sequence of coverings, the
$(k+1)$-graph algebra $C^*(\tgrphlim(\Lambda_n; p_n))$ is Morita
equivalent to the direct limit $\varinjlim(C^*(\Lambda_n),
\iota_{p_n})$.

\subsection{Coactions and coaction crossed products}

Here we give some background on group coactions on $C^*$-algebras and
coaction crossed products. For a detailed treatment of coactions and
coaction crossed-products, see \cite[Appendix~A]{EKQR}.

Given a locally compact group $G$, we write $C^*(G)$ for the full
group $C^*$-algebra of $G$. We prefer to identify $G$ with its
canonical image in $M(C^*(G))$, but when confusion is likely we use
$s \mapsto u(s)$ for the canonical inclusion of $G$ in $M(C^*(G))$.
If $A$ and $B$ are $C^*$-algebras, then $A \otimes B$ denotes the
spatial tensor product. For a group $G$, we write $\delta_G$ for the
natural comultiplication $\delta_G : C^*(G) \to M(C^*(G) \otimes
C^*(G))$ given by the integrated form of the strictly continuous map
which takes $s \in G$ to $s \otimes s \in \Uu M(C^*(G) \otimes
C^*(G))$.

As in \cite[Definition~A.21]{EKQR}, a \emph{coaction} of a group $G$
on a $C^*$-algebra $A$ is an injective homomorphism $\delta : A \to
M(A \otimes C^*(G))$ satisfying
\begin{itemize}
\item[(1)] the \emph{coaction identity} $(\delta \otimes
    1_G) \circ \delta = (1_A \otimes \delta_G) \circ
    \delta$ (as maps from $A$ to $M(A \otimes C^*(G)
    \otimes C^*(G))$); and
\item[(2)] the \emph{nondegeneracy condition}
    $\overline{\delta(A)(1_A \otimes C^*(G))} = M(A \otimes
    C^*(G))$.
\end{itemize}
As in \cite{KQ1, KQ2}, the nondegeneracy condition~(2) ---
rather than the weaker condition that $\delta$ be a
nondegenerate homomorphism --- is part of our definition of a
coaction (compare with Definition~A.21 and Remark~A.22(3)
of~\cite{EKQR}). Since we will be dealing only with coactions
of compact (and hence amenable) groups, the two conditions are
equivalent in our setting in any case (see
\cite[Lemma~3.8]{Lan}).

Let $\delta : A \to M(A \otimes C^*(G))$ be a coaction of $G$
on $A$. We regard the map which takes $s \in G$ to $u(s) \in
M(C^*(G))$ as an element $w_G$ of $\Uu M(C_0(G) \otimes
C^*(G))$. Given a $C^*$-algebra $D$, A \emph{covariant
homomorphism} of $(A, G, \delta)$ into $M(D)$ is a pair $(\pi,
\mu)$ of homomorphisms $\pi : A \to M(D)$ and $\mu : C_0(G) \to
M(D)$ satisfying the covariance condition:
\[
(\pi \otimes \id_G) \circ \delta(a)
 = (\mu \otimes \id_G)(w_G)(\pi(a) \otimes 1) (\mu \otimes
 \id_G)(w_G)^*
\]
for all $a \in A$.

The coaction crossed-product $A \rtimes_\delta G$ is the universal
$C^*$-algebra generated by the image of a universal covariant
representation $(j_A, j_G)$ of $(A, G, \delta)$ (see
\cite[Theorem~A.41]{EKQR}).

\section{Continuity of coaction
crossed-products}\label{sec:coactions}

In this section, we prove a general result regarding the continuity
of the coaction crossed-product construction. Specifically, consider
a projective system of finite groups $G_n$ and a system of compatible
coactions $\delta^n$ of the $G_n$ on a fixed $C^*$-algebra $A$. We
show that this determines a coaction $\delta$ of the projective limit
$\varprojlim G_n$ on $A$, and that the coaction crossed product of
$A$ by $\delta$ is isomorphic to a direct limit of the coaction
crossed products of $A$ by the $\delta^n$.

The application we have in mind is when $A = C^*(\Lambda)$ is a
$k$-graph algebra, and the $\delta^n$ arise from a system of
skew-products of $\Lambda$ by the $G_n$. We consider this
situation in Section~\ref{sec:k-graphs}.

\begin{thm}\label{thm:ccp continuous}
Let $A$ be a $C^*$-algebra, and let
\[\xymatrix{
\cdots \ar[r]^-{q_{n+1}} &G_{n+1} \ar[r]^-{q_n} &G_n \ar[r] &\cdots
\ar[r]^-{q_1} &G_1 }\] be surjective homomorphisms of finite groups.
For each $n$ let $\d^n$ be a coaction of $G_n$ on $A$. Suppose that
the diagram
\begin{equation}\label{eq:commuting triangle hypothesis}
\xymatrix@C+30pt{ A \ar[r]^-{\d^{n+1}} \ar[dr]_{\d^n} &M(A\otimes
C^*(G_{n+1})) \ar[d]^{\id\otimes q_n}
\\
&M(A\otimes C^*(G_n))}
\end{equation}
commutes for each $n$.

For each $n$, write $Q_n$ for the canonical surjective homomorphism
of $\varprojlim (G_m, q_m)$ onto $G_n$; write $q^*_n : C(G_n) \to
C(G_{n+1})$ for the induced map $q^*_n(f) := f \circ q_n$; and write
$J_n$ for the homomorphism $J_n := j^{\d^{n+1}}_A\times
(j_{G_{n+1}}\circ q_n^*)$ from $A\times_{\d^n} G_n$ to
$A\times_{\d^{n+1}} G_{n+1}$.

Then there is a unique coaction $\d$ of $\invlim (G_n,q_n)$ on $A$
such that:
\begin{enumerate}
\item the diagrams
\[\xymatrix{
A \ar[r]^-\d \ar[dr]_{\d^n} &M(A\otimes C^*(\invlim G_n))
\ar[d]^{\id\otimes Q_n}
\\
&M(A\otimes C^*(G_n)) }\] commute; and

\item $A\times_\d \invlim (G_n, q_n) \cong \dirlim
    (A\times_{\d^n} G_n, J_n)$.
\end{enumerate}
\end{thm}

\begin{rmk}
In diagram~\eqref{eq:commuting triangle hypothesis} we could
replace $M(A \otimes C^*(G_n))$ with $A \otimes C^*(G_n)$ and
$M(A \otimes C^*(G_{n+1}))$ with $A \otimes C^*(G_{n+1})$
because $G_n, G_{n+1}$ are discrete.
\end{rmk}

\begin{proof}[Proof of Theorem~\ref{thm:ccp continuous}]
Put\label{page:Jn}
\begin{align*}
G&=\invlim G_n\\
B_n&=A\times_{\d^n} G_n\\
J_n&=j^{\d^{n+1}}_A\times (j_{G_{n+1}}\circ q_n^*):B_n\to B_{n+1} \\
B&=\dirlim (B_n,J_n)\\
K_n&=\text{ the canonical embedding }B_n\to B.
\end{align*}
We aim to apply Landstad duality~\cite{qui:landstad}: we will show
that $B$ is of the form $C\times_\d G$ for some coaction $(C,G,\d)$,
and then we will show that we can take $C=A$. To apply
\cite{qui:landstad} we need:
\begin{itemize}
\item an action $\a$ of $G$ on $B$, and
\item a nondegenerate homomorphism $\m:C(G)\to M(B)$
which is $\rt-\a$ equivariant, where $\rt$ is the action of $G$
on $C(G)$ by right translation.
\end{itemize}
Then \cite{qui:landstad} will provide a coaction $(C,G,\d)$ and an
isomorphism
\[\t:B\iso C\times_\d G\]
such that
\[\t\circ\m=j_G\and \t(B^\a)=j_C(C).\]
This is simpler than the general construction of \cite{qui:landstad},
because our group $G$ is compact (and then we are really using
Landstad's unpublished characterisation \cite{lan:dualcompact} of
crossed products by coactions of compact groups).

We begin by constructing the action $\a$: for each $s\in G$ the
diagrams
\[\xymatrix@C+30pt{
B_{n+1} \ar[r]^-{\what{\d^{n+1}}_{Q_{n+1}(s)}} &B_{n+1}
\\
B_n \ar[u]^{J_n} \ar[r]_-{\what{\d^n}_{Q_n(s)}} &B_n \ar[u]_{J_n}
}\]
commute because
\begin{align*}
\what{\d^{n+1}}_{Q_{n+1}(s)}\circ J_n\circ j^{\d^n}_A
&=\what{\d^{n+1}}_{Q_{n+1}(s)}\circ j^{\d^{n+1}}_A
\\&=j^{\d^{n+1}}_A
\\&=J_n\circ j^{\d^n}_A
\\&=J_n\circ \what{\d^n}_{Q_n(s)}\circ j^{\d^n}_A
\end{align*}
and
\begin{align*}
\what{\d^{n+1}}_{Q_{n+1}(s)}\circ J_n\circ j_{G_n}
&=\what{\d^{n+1}}_{Q_{n+1}(s)}\circ j_{G_{n+1}}\circ q_n^*
\\&=j_{G_{n+1}}\circ \rt_{Q_{n+1}(s)}\circ q_n^*
\\&=j_{G_{n+1}}\circ q_n^*\circ \rt_{q_n\circ Q_{n+1}(s)}
\\&=J_n\circ j_{G_n}\circ \rt_{Q_n(s)}
\\&=J_n\circ \what{\d^n}_{Q_n(s)}\circ j_{G_n}.
\end{align*}
Thus, because the $\what{\d^n}_{Q_n(s)}$ are automorphisms, by
universality there is a unique automorphism $\a_s$ such that the
diagrams
\[\xymatrix@C+20pt{
B \ar@{-->}[r]^-{\a_s} &B
\\
B_n \ar[r]_-{\what{\d^n}_{Q_n(s)}} \ar[u]^{K_n} &B_n
\ar[u]_{K_n} }\] commute. It is easy to check that this gives a
homomorphism $\a:G\to \aut B$. We verify continuity: each
function $s\mapsto \a_s(b)$ for $b\in B$ is a uniform limit of
functions of the form  $s\mapsto \a_s\circ K_n(b)$ for $b\in
B_n$. But we have
\[\a_s\circ K_n(b)
=K_n\circ \what{\d^n}_{Q_n(s)}(b),\] which is continuous since
$K_n$, $Q_n$, and $t\mapsto \what{\d^n}_t(b):G_n\to B_n$ are.

We turn to the construction of the nondegenerate homomorphism $\m$:
first note that the increasing union $\bigcup_nQ_n^*(C(G_n))$ is
dense in $C(G)$ by the Stone-Weierstrass Theorem, and it follows that
there is an isomorphism
\[C(G)\cong \dirlim (C(G_n),q_n^*)\]
taking $Q_n^*$ to the canonical embedding. We have a compatible
sequence of nondegenerate homomorphisms
\[\xymatrix@C+20pt{
C(G_{n+1}) \ar[r]^-{j_{G_{n+1}}} &M(B_{n+1})
\\
C(G_n) \ar[u]^{q_n^*} \ar[r]_-{j_{G_n}} &M(B_n) \ar[u]_{J_n}, }\] so
by universality there is a unique homomorphism $\m$ making the
diagrams
\[\xymatrix{
C(G) \ar@{-->}[r]^-\m &M(B)
\\
C(G_n) \ar[u]^{Q_n^*} \ar[r]_-{j_{G_n}} &M(B_n) \ar[u]_{K_n} }\]
commute. Moreover, $\m$ is nondegenerate since $K_n$ and $j_{G_n}$
are.

We now have $\a$ and $\m$, and the equivariance
\[\a_s\circ\m=\m\circ \rt_s\]
follows from
\begin{align*}
\a_s\circ\m\circ Q_n^* &=\a_s\circ K_n\circ j_{G_n }
\\&=K_n\circ \what{\d^n}_{Q_n(s)}\circ j_{G_n}
\\&=K_n\circ j_{G_n}\circ \rt_{Q_n(s)}
\\&=\m\circ Q_n^*\circ \rt_{Q_n(s)}
\\&=\m\circ \rt_s\circ Q_n^*.
\end{align*}
Thus we can apply \cite{qui:landstad} to obtain a coaction
$(C,G,\d)$ and an isomorphism
\[\t:B\iso C\times_\d G\]
such that
\[\t\circ\m=j_G\and \t(B^\a)=j_C(C).\]
We want to take $C=A$. Note that we have a compatible sequence of
nondegenerate homomorphisms
\[\xymatrix@C+20pt{
A \ar[r]^-{j^{\d^{n+1}}_A} \ar[dr]_{j^{\d^n}_A} &B_{n+1}
\\
&B_n \ar[u]_{J_n}, }\] so by universality there is a unique
homomorphism $j$ making the diagrams
\[\xymatrix{
A \ar[r]^-j \ar[dr]_{j^{\d^n}_A} &B
\\
&B_n \ar[u]_{K_n} }\] commute. Moreover, $j$ is injective and
nondegenerate since $K_n$ and $j^{\d^n}_A$ are. Because $j$, $j_C$,
and $\t$ are faithful, to show that we can take $C=A$ it suffices to
show that
\[j(A)=B^\a.\]
We have
\[j(A)\subset B^\a\]
because
\begin{align*}
\a_s\circ j &=\a_s\circ K_n\circ j^{\d_n}_A
\\&=K_n\circ \what{\d^n}_{Q_n(s)}\circ j^{\d_n}_A
\\&=K_n\circ j^{\d_n}_A
\\&=j.
\end{align*}
For the opposite containment, let $b\in B^\a$. There is a
sequence $b_n\in B_n$ such that $K_n(b_n)\to b$. The functions
$s\mapsto \a_s\circ K_n(b_n)$ converge uniformly to the
function $s\mapsto \a_s(b)$, so
\[\int_G \a_s\circ K_n(b_n)\,ds\to \int_G \a_s(b)\,ds=b.\]
We have
\[\int_G \a_s\circ K_n(b_n)\,ds
=\int_G K_n\circ \what{\d^n}_{Q_n(s)}(b_n)\,ds =K_n\left(\int_G
\what{\d^n}_{Q_n(s)}(b_n)\,ds\right).\] Since
\[\int_G \what{\d^n}_{Q_n(s)}(b_n)\,ds
\in B_n^{\what{\d^n}}=j^{\d^n}_A(A),\] we conclude that
\[b\in K_n\circ j^{\d^n}_A(A)=j(A).\]
Therefore we can take $C=A$, so that we have a coaction $(A,G,\d)$
and an isomorphism
\[\t:B\iso A\times_\d G\]
such that
\[\t\circ\m=j_G.\]

We have proved (ii). For (i), we calculate:
\begin{align*}
(j^\d_A\otimes\id)\circ(\id\otimes Q_n)\circ\d &=(\id\otimes
Q_n)\circ
(j^\d_A\otimes\id)\circ\d
\\&=(\id\otimes Q_n)\circ\ad(j_G\otimes\id)(w_G)\circ(j^\d_A\otimes
1)
\\&=\ad(\id\otimes Q_n)\bigl((j_G\otimes\id)(w_G)\bigr)
\circ(\id\otimes Q_n)\circ(j^\d_A\otimes 1)
\\&=\ad(j_G\otimes\id)\bigl((\id\otimes Q_n)(w_G)\bigr)
\circ(j^\d_A\otimes 1)
\\&=\ad(j_G\otimes\id)\bigl((Q_n^*\otimes\id)(w_{G_n })\bigr)
\circ(j^\d_A\otimes 1)
\\&=\ad(j_G\circ Q_n^*\otimes\id)(w_{G_n}) \circ(j^\d_A\otimes 1)
\\&=\ad(\t\circ K_n\circ j_{G_n}\otimes\id)(w_{G_n}) \circ(\t\circ K_n\circ
j^{\d^n}_A\otimes 1)
\\&=(\t\circ K_n\otimes\id)\circ\ad(j_{G_n}\otimes\id)(w_{G_n})\circ(j^{\d^n}_A\otimes
1)
\\&=(\t\circ K_n\otimes\id)\circ(j^{\d^n}_A\otimes\id)\circ\d^n
\\&=(\t\circ K_n\circ j^{\d^n}_A\otimes\id)\circ\d^n
\\&=(j^\d_A\otimes\id)\circ\d^n.
\end{align*}
Since $j^\d_A$ is faithful, we therefore have $(\id
\otimes Q_n)\circ \d = \d^n$.
\end{proof}

The following application of Theorem~\ref{thm:ccp continuous}
motivates the work of the following sections.

\begin{example}\label{eg:BD}
Let $A = C(\T) = C^*(\Z)$, and let $z$ denote the canonical
generating unitary function $z \mapsto z$. For $n \in \N$, let
$G_n := \Z/2^{n-1}\Z$ be the cyclic group of order $2^{n-1}$.
We write $1$ for the canonical generator of $G_n$ and $0$ for
the identity element. Let $g \mapsto u_n(g)$ denote the
canonical embedding of $G_n$ into $C^*(G_n)$. Define $q_n :
G_{n+1} \to G_n$ by $q_n(m) := m\ ({\rm mod}\ 2^{n-1})$, and
write $q_n$ also for the homomorphism $q_n : C^*(G_{n+1}) \to
C^*(G_n)$ satisfying $q_n(u_{n+1}(g)) = u_n(q_n(g))$. For each
$n$, let $\delta^n$ be the coaction of $G_n$ on $A$ determined
by $\delta^n(z) := z \otimes u_n(1)$.

Let $g \mapsto u(g)$ denote the canonical embedding of
$\varprojlim G_n$ as unitaries in the multiplier algebra of
$C^*(\varprojlim G_n)$. The coaction $\delta$ of $\varprojlim
G_n$ on $A$ described in Theorem~\ref{thm:ccp continuous} is
the one determined by $\delta(z) := z \otimes u(1,1,\dots)$;
the corresponding coaction crossed-product is known to be
isomorphic to the Bunce-Deddens algebra of type $2^\infty$
(see, for example, \cite[8.4.4]{fill}).
\end{example}

\section{Coverings of skew-products}\label{sec:k-graphs}
In this section and the next, we adopt the following notation and
assumptions.

\begin{ntn}\label{ntn:standing}
Let $\Lambda$ be a connected row-finite $k$-graph with no sources.
Fix a vertex $v \in \Lambda^0$, and denote by $\pi\Lambda$ the
fundamental group $\pi_1(\Lambda,v)$ of $\Lambda$ with respect to
$v$. Fix a cocycle $c : \Lambda \to \pi\Lambda$ such that the skew
product $\Lambda \times_c \pi\Lambda$ is isomorphic to the universal
covering $\Omega_\Lambda$ of $\Lambda$ (such a cocycle exists by
\cite[Corollary~6.5]{PQR2}).

Fix a descending chain of finite-index normal subgroups
\begin{equation}\label{eq:subgroups}
\dots \lhd H_{n+1} \lhd H_n \lhd \dots \lhd H_1 := \pi\Lambda.
\end{equation}
For each $n$, let $G_n := \pi\Lambda/H_n$, and let $q_n : G_{n+1} \to
G_n$ be the induced homomorphism
\[
q_n(gH_{n+1}) := gH_n.
\]
Then
\[\xymatrix{
\cdots \ar[r]^-{q_{n+1}} &G_{n+1} \ar[r]^-{q_n} &G_n \ar[r] &\cdots
\ar[r]^-{q_1} &G_1 := \{e\} }
\]
is a chain of surjective homomorphisms of finite groups. Let $G$
denote the projective limit group $\varprojlim(G_n, q_n)$.

For each $n$, let $c_n : \Lambda \to G_n$ be the induced cocycle
$c_n(\lambda) = c(\lambda)H_n$, and let
\[
\Lambda_n := \Lambda \times_{c_n} G_n
\]
be the skew-product $k$-graph. Define covering maps $p_n :
\Lambda_{n+1} \to \Lambda_n$ by $p_n(\lambda, g) := (\lambda,
q_n(g))$.

As in \cite[Theorem~7.1(1)]{PQR2}, for each $n$ there is a coaction
$\delta^n : C^*(\Lambda) \to C^*(\Lambda) \otimes C^*(G_n)$
determined by $\delta^n(s_\lambda) := s_\lambda \otimes
c_n(\lambda)$. Denote by $J_n$ the inclusion
\[
J_n := j^{\delta^{n+1}}_{C^*(\Lambda)}\times (j_{G_{n+1}}\circ q_n^*):
C^*(\Lambda)
\times_{\delta^n} G_n\to C^*(\Lambda) \times_{\delta^{n+1}} G_{n+1}
\]
described in Theorem~\ref{thm:ccp continuous}(ii).

As in \cite[Theorem~7.1(ii)]{PQR2}, for each $n$ there is an
isomorphism $\phi_n$ of $C^*(\Lambda_n) = C^*(\Lambda \times_{c_n}
G_n)$ onto $C^*(\Lambda) \times_{\delta^n} (G_n)$ which satisfies
$\phi_n(s_{(\lambda, g)}) := (s_\lambda, g)$.
\end{ntn}

\begin{example}[Example~\ref{eg:BD} Continued]\label{eg:BDii}
Let $\Lambda$ be the path category of the directed graph $B_1$
consisting of a single vertex $v$ and a single edge $f$ with $r(f) =
s(f) = v$. Note that as a category, $\Lambda$ is isomorphic to $\N$,
and the degree functor is then the identity function from $\N$ to
itself.

Then $\pi\Lambda$ is the free abelian group generated by the homotopy
class of $f$, and so is isomorphic to $\Z$. We define a functor $c :
\Lambda \to \Z$ by $c(f) = 1$.

For each $n$, let $H_n := 2^{n-1}\Z \subset \Z$, so that $\dots \lhd
H_{n+1} \lhd H_n \lhd \dots \lhd H_1 := \pi\Lambda$ is a descending
chain of finite-index normal subgroups. For each $n$, $G_n := \Z /
H_n$ is the cyclic group of order $2^{n-1}$, and $q_n : G_{n+1} \to
G_n$ is the quotient map described in Example~\ref{eg:BD}. The
induced cocycle $c_n : \Lambda \to G_n$ obtained from $c$ is
determined by $c_n(f) = 1 \in \Z/2^{n-1} \Z$.

For $p \in \N$, let $C_p$ denote the simple cycle graph with
$p$ vertices: $C_p^0 := \{v^p_j : j \in \Z/p\Z\}$ and $C_p^1 :=
\{e^p_j : j \in \Z/p\Z\}$, where $r(e^p_i) = v^p_i$ and
$s(e^p_i) = v^p_{i + 1\mod p}$. For each $n$, the skew-product
graph $\Lambda_n := \Lambda \times_{c_n} G_n$ is isomorphic to
the path-category of $C_{2^{n-1}}$. The associated covering map
$p_n : \Lambda_{n+1} \to \Lambda_n$ corresponds to the
double-covering of $C_{2^{n-1}}$ by $C_{2^n}$ satisfying
$v^{2^n}_i \mapsto v^{2^{n-1}}_{i\mod 2^{n-1}}$ and $e^{2^n}_i
\mapsto e^{2^{n-1}}_{i\mod 2^{n-1}}$.

Modulo a relabelling of the generators of $\N^2$, the $2$-graph
$\tgrphlim(\Lambda_n, p_n)$ obtained from this data as in
\cite{KPS} (see Section~\ref{sec:covering systems}) is
isomorphic to the $2$-graph of \cite[Example~6.7]{PRRS}.
Combining this with the final observation of
Example~\ref{eg:BD}, we obtain a new proof that the
$C^*$-algebra of this $2$-graph is Morita equivalent to the
Bunce-Deddens algebra of type $2^\infty$ (see
\cite[Example~6.7]{PRRS} for an alternative proof).
\end{example}

\begin{thm}\label{thm:tower cong ccp}
Adopt the notation and assumptions~\ref{ntn:standing}. Taking $A :=
C^*(\Lambda)$, the coactions $\delta^n$ and the quotient maps $q_n$
make the diagrams~\eqref{eq:commuting triangle hypothesis} commute.
Let $\delta$ denote the coaction of $G := \invlim(G_n, q_n)$ on
$C^*(\Lambda)$ obtained from Theorem~\ref{thm:ccp continuous}. Let
$P_0$ denote the projection $\sum_{v \in \Lambda^0} s_v$ in the
multiplier algebra of $C^*(\tgrphlim(\Lambda_n, p_n))$. Then $P_0$ is
full and
\[
P_0 C^*(\tgrphlim(\Lambda_n, p_n)) P_0 \cong C^*(\Lambda)
\times_\delta G.
\]
\end{thm}

To prove this theorem, we first show that in the setting described
above, the inclusions of $k$-graph algebras induced from the
coverings $p_n : \Lambda_{n+1} \to \Lambda_n$ as in \cite{KPS} are
compatible with the inclusions of coaction crossed products induced
from the quotient maps $q_n : G_{n+1} \to G_n$.

\begin{lem}\label{lem:well-behaved} With the notation and
assumptions~\ref{ntn:standing}, fix $n \in \N$, and let $\iota_{p_n}$
be the inclusion of $C^*(\Lambda_n)$ into $C^*(\Lambda_{n+1})$
obtained from \cite[Proposition~3.3(iv)]{KPS}. Then the inclusion
$\iota_n$ and the isomorphisms $\phi_n, \phi_{n+1}$ of
Notation~\ref{ntn:standing} make the following diagram commute.
\[
\xymatrix{
C^*(\Lambda_n) \ar[r]^-{\iota_{p_n}} \ar[d]^{\phi_n} &
C^*(\Lambda_{n+1}) \ar[d]^{\phi_{n+1}} \\
C^*(\Lambda) \times_{\delta^n} G_n \ar[r]^-{\iota_n} & C^*(\Lambda)
\times_{\delta^{n+1}} G_{n+1} }
\]
\end{lem}
\begin{proof}
By definition, we have
\[
\iota_{p_n}(s_{(\lambda, gH_n)}) = \sum_{p(\lambda', g'H_{n+1}) =
(\lambda, gH_n)} s_{(\lambda', g'H_{n+1})}.
\]
By definition of $p_n$, this becomes
\[
\iota_{p_n}(s_{(\lambda, gH_n)}) = \sum_{\{g'H_{n+1} \in G_{n+1} :
g'H_n=gH_n\}} s_{(\lambda, g'H_{n+1})}.
\]
Hence
\[
\phi_{n+1} \circ \iota_{p_n}(s_{(\lambda, gH_n)}) =
\sum_{\{g'H_{n+1}
\in
G_{n+1} : g'H_n=gH_n\}} (s_{\lambda}, g'H_{n+1}).
\]
But this is precisely $\iota(\phi_n(s_{(\lambda,gH_n)}))$ by
definition of $\iota$ and $\phi_n$.
\end{proof}

\begin{cor}\label{cor:tower cong ccp dir lim}
With the notation and assumptions~\ref{ntn:standing}, let $P_0$
denote the projection $\sum_{v \in \Lambda^0} s_v$ in the multiplier
algebra of $C^*(\tgrphlim(\Lambda_n, p_n))$. Then $P_0$ is full, and
\[
P_0 C^*(\tgrphlim(\Lambda_n, p_n))P_0 \cong \varinjlim (C^*(\Lambda)
\times_{\delta^n} G_n, \iota_n).
\]
\end{cor}
\begin{proof}
Equation~(3.2) of~\cite{KPS} implies that $P_0
C^*(\tgrphlim(\Lambda_n, p_n)) P_0$ is isomorphic to
$\varinjlim(C^*(\Lambda_n), \iota_{p_n})$. The latter is isomorphic
to $\varinjlim (C^*(\Lambda) \times_{\delta^n} G_n, \iota_n)$ by
Lemma~\ref{lem:well-behaved} and the universal property of the direct
limit.
\end{proof}

\begin{proof}[Proof of Theorem~\ref{thm:tower cong ccp}]
It is immediate from the definitions of the maps involved that the
maps $\delta^n$ and $q_n$ make the diagram~\eqref{eq:commuting
triangle hypothesis} commute. The rest of the statement then follows
from Corollary~\ref{cor:tower cong ccp dir lim} and
Theorem~\ref{thm:ccp continuous}(ii).
\end{proof}

\section{Simplicity}\label{sec:simplicity}
In this section we frequently embed $\N^k$ into $\N^{k+1}$ as the
subset consisting of elements whose $(k+1)^{\rm st}$ coordinate is
equal to zero. For $n \in \N^k$, we write $(n,0)$ for the
corresponding element of $\N^{k+1}$.

\begin{thm}\label{thm:simple}
Adopt the notation and assumptions~\ref{ntn:standing}. The
$(k+1)$-graph $C^*$-algebra $C^*(\tgrphlim(\Lambda_n, p_n))$ is
simple if and only if the following two conditions are
satisfied:
\begin{enumerate}
\item
each $\Lambda_n$ is cofinal, and
\item\label{item:separable}
whenever $v \in \Lambda^0$, $p \not= q \in \N^k$ satisfy
$\sigma^p(x) = \sigma^q(x)$ for all $x \in v\Lambda^0$, there
exists $x \in v\Lambda^\infty$, $l \in \N^k$ and $N \in \N$ such
that $c_N(x(p, p+l)) \not= c_N(x(q,q+l))$.
\end{enumerate}
\end{thm}

The idea is to prove the theorem by appealing to
\cite[Theorem~3.1]{RobSi}. To do this, we will first describe the
infinite paths in $\tgrphlim(\Lambda_n,p_n)$. We identify
$\invlim(G_n, q_n)$ with the set of sequences $g =
(g_n)^\infty_{n=1}$ such that $q_n(g_{n+1}) = g_n$ for all $n$.

\begin{lem}\label{lem:inf paths}
Adopt the notation and assumptions~\ref{ntn:standing}. Fix $x
\in \Lambda^\infty$ and $g = (g_n)^\infty_{n=1} \in \invlim
(G_n, q_n)$. For each $n \in \N$ there is a unique infinite
path $(x,g_n) \in \Lambda_n^\infty$ determined by $(x,
g_n)(0,m) = (x(0,m), c_n(x(0,m))^{-1}g_n)$ for all $m \in
\N^k$. There is a unique infinite path $x^g \in
(\tgrphlim(\Lambda_n, p_n))^\infty$ such that $x^g(0, (m,0)) =
x(0,m)$ for all $m \in \N^k$ and $x^g(n e_{k+1}) = (x(0), g_n)$
for all $n \in \N$; moreover $\sigma^{ne_{k+1}}(x^g)(0,(m,0)) =
(x, g_n)(0,m)$ for all $m \in \N^k$. Finally, every infinite
path $y \in (\tgrphlim(\Lambda_n,p_n))^\infty$ is of the form
$\sigma^{ne_{k+1}}(x^g)$ for some $n \in \N$, $x \in
\Lambda^\infty$ and $g \in \invlim (G_n, q_n)$.
\end{lem}
\begin{proof}
That the formula given determines unique infinite paths $(x,g_n)$, $n
\in \N$ follows from \cite[Remarks~2.2]{KP}. That there is a unique
infinite path $x^g$ such that $x^g(0, (m,0)) = x(0,m)$ for all $m \in
\N^k$ and $x^g(n e_{k+1}) = (x(0), g_n)$ for all $n \in \N$ follows
from the observation that for each $n \in \N$ there is a unique path
\[
\alpha = \alpha_{g,n} := e(x(0), g_1) e(x(0), g_2) \cdots e(x(0),
g_n)
\]
with $d(\alpha_{g,n}) = ne_{k+1}$, $r(\alpha) = x(0) \in \Lambda^0$
and $s(\alpha) = (x(0), g_n) \in \Lambda_n^0$, and that for each $m
\in \N^k$,
\begin{align*}
\alpha (x, g_n)(0,m) = x(0,m)& e(x(m), c_1(x(0,m))^{-1}g_1) \\ &\cdots
e(x(m),
c_n(x(0,m))^{-1}g_n)
\end{align*}
is the unique minimal common extension of $x(0,m)$ and $\alpha$. This
also establishes the assertion that $\sigma^{ne_{k+1}}(x^g)(0,(m,0))
= (x, g_n)(0,m)$ for all $m \in \N^k$.

For the final assertion, fix $y \in
(\tgrphlim(\Lambda_n,p_n))^\infty$. We must have $y(0) = (v, g_n)$
for some $v \in \Lambda^0$, $g_n \in G_n = \pi\Lambda/H_n$ and $n \in
\N$. Let $x \in \Lambda_n^\infty$ be the infinite path determined by
$x(0,m) := y(0, (m,0))$ for all $m \in \N^k$. By definition of
$\Lambda_n = \Lambda \times_{c_n} G_n$, we have $x(0,m) := (\alpha_m,
c_n(\alpha_m)^{-1}g_n)$ where each $\alpha_m \in v\Lambda^m$ and $g$
is the element of $\pi\Lambda$ such that $y(0) = v(g_n)$ as above.
There is then an infinite path in $x' \in \Lambda^\infty$ determined
by $x'(0,m) = \alpha_m$ for all $m \in \N^k$. For $n > i \ge 1$,
inductively define $g_i := q_i(g_{i+1})$, and for $n < i$ let $g_i$
be the unique element of $G_i$ such that $y((i-n)e_{k+1}) = (v,
g_i)$; that such $g_i$ exist follows from the definition of
$\tgrphlim(\Lambda_n,p_n)$. Then $g := (g_i)^\infty_{i=1}$ is an
element of $\invlim(G_n, q_n)$ by definition, and routine
calculations using the definitions of the $\Lambda_n$ show that $x =
\sigma^{ne_{k+1}}((x')^g)$.
\end{proof}

\begin{lem}\label{lem:tgrph cofinal}
Adopt the notation and assumptions~\ref{ntn:standing}. Then the
$(k+1)$-graph $\tgrphlim(\Lambda_n,p_n)$ is cofinal if and only if
each $\Lambda_n$ is cofinal.
\end{lem}
\begin{proof}
First suppose that each $\Lambda_n$ is cofinal. Fix $y \in
\tgrphlim(\Lambda_n,p_n)$ and $w \in \tgrphlim(\Lambda^0)$. By
Lemma~\ref{lem:inf paths}, we have $y = \sigma^{i_0
e_{k+1}}(x^g)$ for some $g = (g_n)^\infty_{n=1} \in
\invlim(G_n, q_n)$, some $i_0 \in \N$ and some $x \in
\Lambda^\infty$. We must show that $w (\tgrphlim(\Lambda_n,
p_n)) y(q) \not= \emptyset$ for some $q$. We have $w \in
\Lambda_m^0$ for some $m \in \N$, so $w = (w', h)$ for some $h
\in G_m$. If $m < i_0$, fix any $h' \in \pi\Lambda$ such that
$h'H_{i_0} = h$, and note that $w (\tgrphlim(\Lambda_n, p_n))
(w', hH_{i_0})$ is nonempty, so that it suffices to show that
$(w', h'H_{i_0})(\tgrphlim(\Lambda_n, p_n)) y(q) \not=
\emptyset$ for some $q$. That is to say, we may assume without
loss of generality that $m \ge i_0$. But now $w \in
\Lambda_m^0$ and $\sigma^{(0, \dots, 0, m-i_0)}(y) \in
(\tgrphlim(\Lambda_n,p_n))^\infty$ with $r(y) \in
\Lambda_{i_0}^0$. Since $\Lambda_n$ is cofinal, we have $w
\Lambda_{i_0} (x, g_m)(q) \not= \emptyset$ for some $q \in
\N^k$ (recall that $x, (g_i)^\infty_{i=1}$ are such that $y =
\sigma^{i_0 e_{k+1}}(x^g)$). By definition, $(x, g_m)(q) =
y(q_1, \dots, q_k, m-i_0)$ and this shows that $w
(\tgrphlim(\Lambda_n, p_n)) y(q) \not= \emptyset$ for $q =
(q_1, \dots, q_k, m-n)$.

Now suppose that $\tgrphlim(\Lambda_n,p_n)$ is cofinal. Fix $n
\in \N$ and a vertex $w$ and an infinite path $x$ in
$\Lambda_n$. Then $x(0) = (v, gH_n)$ for some $v \in
\Lambda^0$, $g \in \pi\Lambda$. There are paths $\alpha_m \in
\Lambda_n^m$, $m \in \N^k$ determined by $x(0,m) = (\alpha_m,
c_n(\alpha_m)^{-1} g H_n)$; there is then an infinite path $x'
\in \Lambda^\infty$ such that $x'(0,m) = \alpha_m$ for all $m$.
Let $g_i := gH_i$ for all $i \in \N$. By abuse of notation we
denote by $g$ the element $(gH_i)^\infty_{i=1}$ of
$\invlim(G_n, q_n)$. Let $y = \sigma^n((x')^g)$ be the infinite
path of $\tgrphlim(\Lambda_n, p_n)$ provided by
Lemma~\ref{lem:inf paths}. As $\tgrphlim(\Lambda_n, p_n)$ is
cofinal, we may fix a path $\lambda \in
\tgrphlim(\Lambda_n,p_n)$ such that $r(\lambda) = w$ and
$s(\lambda)$ lies on $y$. By definition of $y$, there exist $n'
\ge n$  and $m \in \N^k$ such that $s(\lambda) = (x'(m),
c_{n'}(\alpha_m)^{-1}g_{n'})$. We then have $d(\lambda)_{k+1} =
n' - n$, and we may factorise $\lambda = \lambda' \lambda''$
where $d(\lambda') = d(\lambda) - (n' - n)e_{k+1}$ and
$d(\lambda'') = (n' - n)e_{k+1}$. By construction of
$\tgrphlim(\Lambda_n, p_n)$, if $d(\mu) = je_{k+1}$ and $s(\mu)
= (v, gH_n) \in \Lambda_n^0$ then $n \ge j$ and $r(\mu) = (v,
gH_{n-j}) \in \Lambda^0_{n-j}$. In particular,
\[
s(\lambda') = r(\lambda'') = (x'(m), c_{n}(\alpha_m)^{-1}g_{n}) =
x(m),
\]
so $w \Lambda_n x(m) \not=\emptyset$.
\end{proof}

\begin{lem}\label{lem:tgrph aperiodic}
Adopt the notation and assumptions~\ref{ntn:standing}. Then the
$(k+1)$-graph $\tgrphlim(\Lambda_n,p_n)$ has no local periodicity if
and only if it satisfies condition~\ref{item:separable} of
Theorem~\ref{thm:simple}.
\end{lem}
\begin{proof}
First suppose that condition~\ref{item:separable} of
Theorem~\ref{thm:simple} holds. Fix a vertex $v \in
(\tgrphlim(\Lambda_n,p_n))^0$ and $p \not= q \in \N^{k+1}$. So
$v \in \Lambda_n^0$ for some $n$, and $v$ therefore has the
form $v = (w, gH_n)$ for some $w \in \Lambda^0$ and $g \in
\pi\Lambda$. We must show that there exists $x \in
v(\tgrphlim(\Lambda_n,p_n))^\infty$ such that $\sigma^p(x)
\not= \sigma^q(x)$.

We first consider the case where $p_{k+1} \not= q_{k+1}$. By
construction of the tower graph $\tgrphlim(\Lambda_n,p_n)$,
this forces the vertices $x(p)$ and $x(q)$ to lie in distinct
$\Lambda_n$ for any $x \in v(\tgrphlim(\Lambda_n,p_n))^\infty$;
in particular they cannot be equal.

Now suppose that $p_{k+1} = q_{k+1}$. If every $x \in
v(\tgrphlim(\Lambda_n,p_n))^\infty$ satisfies $\sigma^p(x) =
\sigma^q(x)$, then for any $\alpha \in
v(\tgrphlim(\Lambda_n,p_n))^{p_{k+1} e_{k+1}}$ and any $y \in
s(\alpha)(\tgrphlim(\Lambda_n,p_n))^\infty$, we have
$\sigma^p(\alpha y) = \sigma^q(\alpha y)$; that is,
\[
\sigma^{p - p_{k+1} e_{k+1}}(y) = \sigma^{q - q_{k+1} e_{k+1}}(y)
\quad\text{ for all $y \in
s(\alpha)(\tgrphlim(\Lambda_n,p_n))^\infty$.}
\]
So we may assume without loss of generality that $p_{k+1} = q_{k+1} =
0$. Write $p'$ and $q'$ for the elements of $\N^k$ whose entries are
the first $k$ entries of $p$ and $q$.

We have $v \in \Lambda_n$ for some $n$, so there exists $w \in
\Lambda^0$ and $g \in \pi\Lambda$ such that $v = (w, gH_n)$. Suppose
first that there exists $x \in w\Lambda^\infty$ such that
$\sigma^{p'}(x) \not= \sigma^{q'}(x)$, then the infinite path
$(x,gH_n) \in v\Lambda_n^{\infty}$ such that
\[
(x,gH_n)(0,m) := (x(0,m), c_n(x(0,m))^{-1}gH_n)\quad\text{ for all
$m
\in
\N^k$}
\]
also satisfies $\sigma^{p'}((x,gH_n)) \not= \sigma^{q'}((x,gH_n))$.
By Lemma~\ref{lem:inf paths} we may choose an infinite path $y$ such
that $y|_{\N^k \times\{0\}} = (x, gH_n)$, and then $y \in
v(\tgrphlim(\Lambda_n,p_n))^\infty$ satisfies $\sigma^p(y) \not=
\sigma^q(y)$.

Now suppose that every path $x \in w\Lambda^\infty$ satisfies
$\sigma^{p'}(x) = \sigma^{q'}(x)$. Then by
condition~\ref{item:separable} of Theorem~\ref{thm:simple}, we may
fix $x \in w\Lambda^\infty$ and $N \in \N$ such that $c_N(x(0,p'))
\not= c_N(x(0,q'))$. It then follows from the definition of the $c_j$
that $c_j(x(0,p')) \not= c_j(x(0,q'))$ whenever $j \ge N$. So with $j
:= \max\{N,n\}$, we have
\begin{align*}
(x,gH_j)(p') &= (x(p'), c_j(x(0,p'))^{-1}gH_j) \\
&\not= (x(q'), c_j(x(0,q'))^{-1}gH_j) = (x,gH_j)(q').
\end{align*}
There is an element $g = (g_i)^\infty_{i=1}$ of $\invlim(G_n, q_n)$
determined by $g_i := gH_i$ for all $i$. Let $x^g$ be the element of
$(\tgrphlim(\Lambda_n, p_n))^\infty$ determined by $x$ and $g$ as in
Lemma~\ref{lem:inf paths}. Then $(x,gH_n)((j-n)e_{k+1} + p) \not=
(x,gH_n)((j-n)e_{k+1} + q)$, and therefore $x^g$ satisfies
$\sigma^{p}(x^g) \not= \sigma^q(x^g)$ as required. Hence
condition~\ref{item:separable} of Theorem~\ref{thm:simple} implies
that $\tgrphlim(\Lambda_n,p_n)$ has no local periodicity.

To show that if $\tgrphlim(\Lambda_n,p_n)$ has no local periodicity
then condition~\ref{item:separable} of Theorem~\ref{thm:simple}
holds, we prove the contrapositive statement. Suppose that
condition~\ref{item:separable} of Theorem~\ref{thm:simple} does not
hold. Fix $v \in \Lambda^0$ and $p,q \in \N^k$ such that $\sigma^p(x)
= \sigma^q(x)$ for all $x \in v\Lambda^\infty$ and $c_n(x(p, p+l)) =
c_n(x(q, q+l))$ for all $n \in \N$, $l \in \N^k$. Then for each $x
\in v\Lambda^\infty$ and each $g = (g_n)^\infty_{n=1} \in
\invlim(G_n, p_n)$, we have $\sigma^p(x, g_n)(0,l) = \sigma^q(x,
g_n)(0,l)$ for all $n \in \N$ and $l \in\N^k$. Hence
Lemma~\ref{lem:inf paths} implies that every $y \in
v(\tgrphlim(\Lambda_n, p_n))^\infty$ satisfies $\sigma^{(p,0)}(y) =
\sigma^{(q,0)}(y)$.
\end{proof}

\begin{proof}[Proof of Theorem~5.1]
Theorem~3.1 of \cite{RobSi} implies that
$C^*(\tgrphlim(\Lambda_n, p_n))$ is simple if and only if
$\tgrphlim(\Lambda_n, p_n)$ is cofinal and has no local
periodicity. The result then follows directly from Lemmas
\ref{lem:tgrph cofinal}~and~\ref{lem:tgrph aperiodic}.
\end{proof}

\section{Projective limit $k$-graphs}\label{sec:topological
k-graphs} Let $(\Lambda_n, \Lambda_{n+1}, p_n)^\infty_{n=1}$ be
a sequence of row-finite coverings of $k$-graphs with no
sources as in Section~\ref{sec:covering systems}. We aim to
show that the sets $(\varprojlim \Lambda_i)^m := \varprojlim
(\Lambda_i^m, p_i)$ under the projective limit topology with
the natural (coordinate-wise) range and source maps specify a
topological $k$-graph (in the sense of Yeend). Moreover, we
show that the associated topological $k$-graph $C^*$-algebra is
isomorphic to the full corner $P_0 C^*(\tgrphlim(\Lambda_n;
p_n)) P_0$ determined by $P_0 := \sum_{v \in \Lambda^0_1} s_v$.
In particular, when the $\Lambda_n$ and $p_n$ are as
in~\ref{ntn:standing}, the $C^*$-algebra of the projective
limit topological $k$-graph is isomorphic to the crossed
product of $C^*(\Lambda)$ by the coaction of the projective
limit of the groups $G_i$ obtained from Theorem~\ref{thm:ccp
continuous}.

Let $(\Lambda_n, \Lambda_{n+1}, p_n)^\infty_{n=1}$ be a sequence of
row-finite coverings of $k$-graphs with no sources. Let
$\varprojlim(\Lambda_i, p_i)$ be the projective limit category,
equipped with the projective limit topology. That is,
$\varprojlim(\Lambda_i, p_i)$ consists of all sequences
$(\lambda_i)^\infty_{i=1}$ such that each $\lambda_i \in \Lambda_i$
and $p_i(\lambda_{i+1}) = \lambda_i$; the structure maps $\tilde{r}$,
$\tilde{s}$, $\tilde{\circ}$ and $\widetilde{\id}$ on
$\varprojlim(\Lambda_i, p_i)$ are obtained by pointwise application
of the corresponding structure maps for $\Lambda$. The cylinder sets
$Z(\lambda_1, \dots, \lambda_j) := \{(\mu_i)^\infty_{i=1} \in
\varprojlim(\Lambda_i, p_i) : \mu_i = \lambda_i\text{ for }1 \le i
\le j\}$, form a basis of compact open sets for a locally compact
Hausdorff topology.

Define $\tilde{d} : \varprojlim(\Lambda_i, p_i) \to \N^k$ by
$\tilde{d}((\lambda_i)^\infty_{i=1}) := d(\lambda_1)$. Since the
$p_i$ are degree-preserving, we have
\[
\tilde{d}((\lambda_i)^\infty_{i=1}) = d(\lambda_i)\quad\text{for all
$i
\ge 1$.}
\]
For fixed $\lambda = (\lambda_i)^\infty_{i=1} \in
\varprojlim(\Lambda_i, p_i)^{m+n}$, the unique factorisation property
for each $\lambda_i$ produces unique elements $\lambda(0,m) :=
(\lambda_i(0,m))^\infty_{i=1} \in \varprojlim(\Lambda_i, p_i)^m$ and
$\lambda(m,n) := (\lambda_i(m,n))^\infty_{i=1} \in
\varprojlim(\Lambda_i, p_i)^n$ such that $\lambda =
\lambda(0,m)\lambda(m,n)$; that is, $(\varprojlim(\Lambda_i, p_i),
\tilde{d})$ is a second-countable small category with a degree
functor satisfying the factorisation property.

The identity $\tilde{d}((\lambda_i)^\infty_{i=1}) = d(\lambda_i)
\quad \text{for all $i \ge 1$}$ implies that $Z(\lambda_1, \dots,
\lambda_j)$ is empty unless $d(\lambda_1) = \dots = d(\lambda_j)$,
and it follows that $\tilde{d}$ is continuous.

We claim that $\tilde{r}$ and $\tilde{s}$ are local homeomorphisms.
To see this, fix a cylinder set $Z(v_1, \dots, v_j) \subset
\varprojlim(\Lambda_i, p_i)^0$, and for $\lambda \in v_1\Lambda_1$
and $2 \le l \le j$, let $v_lp^{-1}_{1,l}(\lambda)$ be the unique
element of $v_l\Lambda_l$ such that $p_1\circ p_2 \circ\cdots\circ
p_{l-1}(v_lp^{-1}_{1,l}(\lambda)) = \lambda$. Then
\[
\tilde{r}^{-1}(Z(v_1, \dots, v_j)) \cap \varprojlim(\Lambda_i,
p_i)^n
:=
\sqcup_{\lambda \in v_1\Lambda_1^n} Z(\lambda,
v_2p_{1,2}^{-1}(\lambda),
\dots, v_jp_{1,j}^{-1}(\lambda))
\]
which is clearly open, showing that $\tilde{r}$ is continuous.
Moreover, this same formula shows that for $\lambda =
(\lambda_i)^\infty_{i=1} \in \varprojlim(\Lambda_i, p_i)$, the
restriction of $\tilde{r}$ to $Z(\lambda_1)$ is a homeomorphism, and
$\tilde{r}$ is a local homeomorphism as claimed. A similar argument
shows that $\tilde{s}$ is also a local homeomorphism.

It is easy to see that the inverse image under composition of the
cylinder set $Z(\lambda_1, \dots, \lambda_j) \in
\varprojlim(\Lambda_i, p_i)^n$ is equal to the disjoint union
\[
\bigsqcup_{p+q = n} Z(\lambda_1(0,p), \dots, \lambda_j(0,p)) \times
Z(\lambda_1(p,q),\dots, \lambda_j(p,q))
\]
of cartesian products of cylinder sets and hence is open, so that
composition is continuous, and it follows that
$(\varprojlim(\Lambda_i, p_i), \tilde{d})$ is a topological $k$-graph
in the sense of Yeend \cite{Y1, Y2}.

Let $\tgrphlim(\Lambda_n; p_n)$ be as described in
Section~\ref{sec:covering systems}, and let $P_0$ denote the
full projection $\sum_{v \in \Lambda^0_1} s_v \in
M(C^*(\tgrphlim(\Lambda_n; p_n)))$. For the following
proposition, we need to describe $P_0 C^*(\tgrphlim(\Lambda_n;
p_n)) P_0$ in detail. For $n \ge m \ge 1$, we write $p_{m,n} :
\Lambda_n \to \Lambda_m$ for the covering map $p_{m,n} := p_m
\circ \dots \circ p_{n-1}$, with the convention that $p_{n,n}$
is the identity map on $\Lambda_n$. For $v \in \Lambda_m^0$,
and $l \le m$, we denote by $\alpha_{l,m}(v)$ the unique path
in $\tgrphlim(\Lambda_n; p_n)^{(m-l)e_{k+1}}$ whose source is
$v$ (and whose range is $p_{l,m}(v)$). In particular,
$\alpha_{1,m}(v)$ the unique path in $\tgrphlim(\Lambda_n;
p_n)^{(m-1)e_{k+1}}$ whose source is $v$ with range in
$\Lambda_1$. For $\lambda \in \Lambda_m$,
\begin{align*}
s_{\alpha_{1,m}(r(\lambda))} s^*_{\alpha_{1,m}(r(\lambda))} s_{p_{1,m}(\lambda)}
  &= s_{\alpha_{1,m}(r(\lambda))} s_\lambda s^*_{\alpha_{1,m}(s(\lambda))} \\
  &= s_{p_{1,m}(\lambda)} s_{\alpha_{1,m}(s(\lambda))} s^*_{\alpha_{1,m}(s(\lambda))}.
\end{align*}
Furthermore, $P_0 C^*(\tgrphlim(\Lambda_n, p_n)) P_0$ is equal to the
closed span
\[
P_0 C^*(\tgrphlim(\Lambda_n, p_n)) P_0
 = \clsp\{s_{\alpha_{1,m}(r(\lambda))} s_\lambda
 s^*_{\alpha_{1,m}(s(\lambda))} :
          m \ge 1, \lambda \in \Lambda_m\}.
\]

\begin{prop}\label{prp:tower and topological k-graph}
Let $(\Lambda_n, \Lambda_{n+1}, p_n)^\infty_{n=1}$ be a
sequence of row-finite coverings of $k$-graphs with no sources,
and let $\tgrphlim(\Lambda_n; p_n)$ be the associated
$(k+1)$-graph as in \cite{KPS}. Let $P_0 := \sum_{v \in
\Lambda^0_1} s_v \in M C^*(\tgrphlim(\Lambda_n; p_n))$. Let
$(\varprojlim(\Lambda_i, p_i), \tilde{d})$ be the topological
$k$-graph defined above. Then there is a unique isomorphism
\[
\pi : P_0 C^*(\tgrphlim(\Lambda_n,p_n)) P_0 \to
C^*(\varprojlim(\Lambda_i,
p_i))
\]
such that for $\lambda \in \Lambda_m$,
\begin{equation}\label{eq:pi def}
\pi(s_{\alpha_{1,m}(r(\lambda))} s_\lambda
s^*_{\alpha_{1,m}(s(\lambda))})
 = \chi_{Z(p_{1,m}(\lambda), p_{2,m}(\lambda), \dots,
 p_{m-1,m}(\lambda),\lambda)}.
\end{equation}
In particular, with the notation and
assumptions~\eqref{ntn:standing}, there is an isomorphism of
the $C^*$-algebra $C^*(\varprojlim(\Lambda_i, p_i))$ of the
topological $k$-graph $\varprojlim(\Lambda_i, p_i)$ with the
coaction crossed-product $C^*(\Lambda) \times_\delta G$.
\end{prop}
\begin{proof}
The final statement will follow from Theorem~\ref{thm:tower cong ccp}
once we establish the first statement.

To prove the first statement we will use Allen's
gauge-invariant uniqueness theorem for corners in $k$-graph
algebras~\cite{A}. We adopt Allen's notation: for $\mu,\nu \in
\Lambda_1^0 \tgrphlim(\Lambda_n; p_n)$, we let $t_{\mu,\nu} :=
s_\mu s^*_\nu \in P_0 C^*(\tgrphlim(\Lambda_n;p_n)) P_0$. The
factorisation property guarantees that for $\mu,\nu \in
\Lambda_1^0 \tgrphlim(\Lambda_n; p_n)$, we can rewrite $\mu =
\alpha_{1,m}(r(\mu'))\mu'$ and $\nu =
\alpha_{1,m}(r(\nu'))\nu'$ for some $m \ge 1$ and $\mu',\nu'
\in \Lambda_m$ with $s(\mu') = s(\nu')$. By
\cite[Corollary~3.7]{A}, there is an isomorphism $\theta$ of
$P_0 C^*(\tgrphlim(\Lambda_n;p_n)) P_0$ onto Allen's universal
algebra $C^*(\tgrphlim(\Lambda_n; p_n), \Lambda_1^0)$ (see
Definition~3.1 and the following paragraphs in \cite{A}) which
satisfies $\theta(t_{\mu,\nu}) = T_{\mu,\nu}$ for all
$\mu,\nu$. It therefore suffices to show that there is an
isomorphism $\psi : C^*(\tgrphlim(\Lambda_n;p_n), \Lambda_1^0)
\to C^*(\varprojlim(\Lambda_i, p_i))$ such that
$\psi(T_{\alpha_{1,m}(r(\mu)) \mu, \alpha_{1,m}(r(\nu))\nu}) =
\chi_{Z(p_{1,m}(\mu), \dots, \mu) *_s
Z(p_{1,m}(\nu),\dots,\nu)}$ for all $m \ge 1$ and $\mu,\nu \in
\Lambda_m$ with $s(\mu) = s(\nu)$; the composition $\pi := \psi
\circ \theta$ clearly satisfies~\eqref{eq:pi def}, and it is
uniquely specified by~\eqref{eq:pi def} because the elements
$\{t_{\alpha_{1,m}(r(\lambda))\lambda,
\alpha_{1,m}(s(\lambda))} : m \ge 1, \lambda \in \Lambda_m\}$
generate $P_0 C^*(\tgrphlim(\Lambda_n;p_n)) P_0$ as a
$C^*$-algebra.

Let $\Gamma$ denote the topological $k$-graph $\varprojlim(\Lambda_i,
p_i)$. Since $\Gamma$ is row-finite and has no sources,
$\partial\Gamma = \Gamma^\infty$. As in \cite{Y1}, for open subsets
$U, V \subset \Gamma$, let $Z_{\GG_\Gamma}(U *_s V, m)$ denote the
set $\{(\mu x, m, \nu x) : \mu \in U, \nu \in V, x \in \Gamma^\infty,
s(\mu) = s(\nu) = r(x)\}$. Then $\GG_\Gamma$ is the locally compact
Hausdorff topological groupoid
\[
\GG_\Gamma = \{(x,m-n,y) : x,y \in \Gamma^\infty, m,n \in \N^k,
\sigma^m(x) = \sigma^n(y)\}
\]
where the $Z_{\GG_\Gamma}(U *_s V, m)$ form a basis of compact open
sets for the topology.

For $m \ge 1$ and $\lambda \in \Lambda_m$, let $U_{m,\lambda} :=
Z(p_{1,m}(\lambda), \dots, \lambda) \subset \Gamma$. So the
$U_{m,\lambda}$ are a basis for the topology on $\Gamma =
\varprojlim(\Lambda_i, p_i)$. Now for $m \ge 1$ and $\mu,\nu \in
\Lambda_m$ with $s(\mu) = s(\nu)$, let
\[
u_{\alpha_{1,m}(r(\mu)) \mu, \alpha_{1,m}(r(\nu))\nu} :=
\chi_{Z(U_{m,\mu}
*_s U_{m,\nu}, d(\mu) - d(\nu))} \in C_c(\GG_\Gamma).
\]

Tedious but routine calculations using the definition of the
convolution product and involution on $C_c(\GG_\Gamma) \subset
C^*(\GG_\Gamma)$ show that
\[
\{u_{\alpha_{1,m}(r(\mu)) \mu,
\alpha_{1,m}(r(\nu))\nu} : m \ge 1, \mu,\nu \in \Lambda_m,
s(\mu) = s(\nu)\}
\]
is a Cuntz-Krieger $(\tgrphlim(\Lambda_n; p_n),
\Lambda_1^0)$-family in $C^*(\GG_\Gamma)$. By the universal
property of $C^*(\tgrphlim(\Lambda_n;p_n), \Lambda_1^0)$ (see
\cite[Section~3]{A}), there therefore exists a homomorphism
$\psi : C^*(\tgrphlim(\Lambda_n;p_n), \Lambda_1^0) \to
C^*(\GG_\Gamma)$ such that
\[
\psi(T_{\alpha_{1,m}(r(\mu)) \mu, \alpha_{1,m}(r(\nu))\nu}) =
u_{\alpha_{1,m}(r(\mu)) \mu, \alpha_{1,m}(r(\nu))\nu}
\]
for each $m,\mu,\nu$. The canonical gauge action $\beta : \T^k
\to \aut(C^*(\GG_\Gamma))$ determined by $\beta_z(f)(x,m,y) :=
z^{m}f(x,m,y)$ satisfies $\psi \circ \gamma_z = \beta_z \circ
\psi$ for all $z \in \T^k$, where $\gamma$ is the gauge action
on $C^*(\tgrphlim(\Lambda_n;p_n), \Lambda_1^0)$.
Proposition~4.3 of~\cite{Y1} shows that each
$u_{\alpha_{1,m}(r(\mu)) \mu, \alpha_{1,m}(r(\mu))\mu}$ is
nonzero, and it follows from the gauge-invariant uniqueness
theorem \cite[Theorem~3.5]{A} that $\psi$ is injective. The
topology on $\GG_\Gamma^{(0)}$ is generated by the collection
of compact open sets $\{U_{m,\lambda} : m \ge 1, \lambda \in
\Lambda_m\}$, and the topology on $\GG_\Gamma$ is generated by
the collection of compact open sets $\{U_{m,\mu} *_s U_{m,\nu}
: m \ge 1, \mu,\nu \in \Lambda_m, s(\mu) = s(\nu)\}$. Since
$C^*(\{u_{\alpha_{1,m}(r(\mu)) \mu, \alpha_{1,m}(r(\nu))\nu} :
m \ge 1, \mu,\nu \in \Lambda_m, s(\mu) = s(\nu)\}) \subset
C^*(\GG_\Gamma)$ contains the characteristic functions of these
sets, it follows that $\psi$ is also onto, and this completes
the proof.
\end{proof}

\begin{rmk}
The final statement of Proposition~\ref{prp:tower and
topological k-graph} suggests that we can regard
$\varprojlim(\Lambda_i, p_i)$ as a skew-product of $\Lambda$ by
$G$.

To make this precise, note that for $\lambda \in \Lambda$,
$c(\lambda) := (c_n(\lambda))^\infty_{n=1}$ belongs to $G$, and $c :
\Lambda \to G$ is then a cocycle. There is a natural bijection
between the cartesian product $\Lambda \times G$ and the topological
$k$-graph $\varprojlim(\Lambda_i, p_i)$, so we may view $\Lambda
\times G$ as a topological $k$-graph by pulling back the structure
maps from $\varinjlim(\Lambda_i, p_i)$. What we obtain coincides with
the natural definition of the skew-product $\Lambda \times_c G$.

With this point of view, we can regard Proposition~\ref{prp:tower and
topological k-graph} as a generalisation of
\cite[Theorem~7.1(ii)]{PQR2} to profinite groups and topological
$k$-graphs: $C^*(\Lambda \times_c G) \cong C^*(\Lambda) \times_\delta
G$.
\end{rmk}

\begin{example}[Example~\ref{eg:BD} continued]
Resume the notation of Examples \ref{eg:BD}~and~\ref{eg:BDii}.
The resulting projective limit $\varprojlim(\Lambda_n, p_n)$ is
the topological $1$-graph $E$ associated to the odometer action
of $\Z$ on the Cantor set as in \cite[Example~2.5(3)]{Y1}. That
is, $E$ can be realised as the skew-product of $B_1^*$ by the
$2$-adic integers $\Z_2$ with respect to the functor $c : B_1^*
\to \Z_2$ determined by $c(f) = (1,1,1,\dots)$, where $f$ is
the loop edge generating $B_1^*$.
\end{example}

\end{document}